\documentclass[12pt, reqno]{article}

\newcommand{\seqnum}[1]{\href{http://www.research.att.com/cgi-bin/access.cgi/as/~njas/sequences/eisA.cgi?Anum=#1}{\underline{#1}}}

\usepackage[usenames]{color}
\usepackage{amsmath, amsthm, amssymb, amsfonts, amscd}
\usepackage{graphicx}
\usepackage[colorlinks=true,
linkcolor=webgreen,
filecolor=webbrown,
citecolor=webgreen]{hyperref}

\definecolor{webgreen}{rgb}{0,.5,0}
\definecolor{webbrown}{rgb}{.6,0,0}

\newtheorem{thm}{Theorem}
\newtheorem{lemma}[thm]{Lemma}

\newtheorem{case}{Case}
\newtheorem{conj}[thm]{Conjecture}

\begin{document}

\begin{center}
\vskip 1cm{\LARGE\bf 
On a Conjecture of Andrica \& Tomescu
}
\vskip 1cm
\large
Blair D. Sullivan\footnote{The submitted manuscript has been authored by a contractor of the U.S. Government under Contract No. DE-AC05-00OR22725. Accordingly, the U.S. Government retains a non-exclusive, royalty-free license to publish or reproduce the published form of this contribution, or allow others to do so, for U.S. Government purposes.}\\
Oak Ridge National Laboratory\\ 
1 Bethel Valley Road, Oak Ridge, TN 37831\\
{\tt sullivanb@ornl.gov}
\end{center}
\vskip .2 in

\begin{abstract}
    Let $S(n)$ be the integer sequence which is the coefficient of $x^{n(n+1)/4}$
    in the expansion of $(1+x)(1+x^2)\dotsm (1+x^n)$ for positive integers $n$ 
    congruent to 0 or 3 mod 4. We prove a conjecture of 
    Andrica and Tomescu \cite{andrica_tomescu} that 
    $S(n)$ is asymptotic to $\sqrt{6/\pi} \cdot 2^n n^{-3/2}$ as $n$ 
    approaches infinity.
\end{abstract}

\bigskip

\section{Introduction}
Let $S(n)$ denote the coefficient of the middle term of the 
expansion of the polynomial $(1+x)(1+x^2)\dotsm(1+x^n)$ 
when $n \equiv$ 0 or 3 mod 4 (otherwise $n(n+1)/4$ is not an integer, 
and the expansion has no middle term). Andrica and Tomescu conjectured that as $n$ 
approaches infinity, $S(n)$ behaves asymptotically like 
$\sqrt{6/\pi} \cdot 2^n n^{-3/2}$. 
More formally, writing $f(n) \sim g(n)$ to denote 
$$\lim_{n\rightarrow\infty} \frac{f(n)}{g(n)} = 1,$$ 
we have
\begin{conj}\label{atconj}[Andrica, Tomescu~\cite{andrica_tomescu}] 
$S(n) \sim \sqrt{6/\pi} \cdot 2^n n^{-3/2}$ for $n \equiv$ $0$ or $3$ mod $4$. 
\end{conj}

From \cite{andrica_tomescu}, one can write $S(n)$ in integral form via Cauchy's formula as
$$S(n)= \frac{2^{n-1}}{\pi}\int_{0}^{2\pi}\cos(t)\cos(2t)\dotsm \cos(nt)\,dt.$$
We will use the Laplace method to estimate
this integral \cite{dB}. Rewriting, we have $S(n) = \frac{2^{n-1}}{\pi}
\int_{0}^{2\pi}f_{n}(t)\,dt$ where $f_{n}(t) = \prod_{k=1}^{n} \cos(kt)$. 
In Section~\ref{lemmasect}, we analyze the behavior of $f_n(t)$ and note a 
technical lemma needed for the main proof of Conjecture~\ref{atconj}, which is presented in 
Section~\ref{mainsect}.

\section{Behavior of $f_n(t)$}\label{lemmasect}

\begin{lemma}\label{uniform_zero_lemma}
  Let $0 < \varepsilon < 1/4$, and $f_{n}(t) = \prod_{k=1}^{n} \cos(kt)$. 
  Then $\int_{n^{-(3/2-\varepsilon)} < |t| < \pi/2} |f_n(t)| \,dt = o(n^{-3/2})$ as $n \to \infty$.  
  \end{lemma}

  \begin{proof}
    
    We break the integral into three pieces based on the value of $|t|$.

    \begin{case}$n^{-(\frac{3}{2}-\varepsilon)} \leq |t|  \leq \frac{1}{n}$:\end{case}
    Since $\cos(x) = \cos(-x)$, and $\cos$ is a monotone decreasing function on $[0,\pi]$,
    $f_n(t) = f_n(-t)$ is also monotone decreasing for $t \in [0,1/n]$, and it suffices to 
    give an appropriate upper bound on $f_n(n^{-(3/2-\varepsilon)})$. 

    Since we need  $\int_{n^{-(3/2-\varepsilon)} < |t| < \pi/2} f_n(t)\,dt = o(n^{-3/2})$, 
    given that $0 < \varepsilon < 1/4,$
    it suffices to show that for a constant $c >0$,  
    $$f_n(n^{-(3/2-\varepsilon)}) \leq \exp(-cn^{2\varepsilon}(1+o(1))).$$ 
    
    Using the Taylor series expansion, we know $\cos(kt) \leq 1 - \frac{(kt)^{2}}{2!} + \frac{(kt)^{4}}{4!}$. 
    Substitution then yields 
    $$f_n(t) = \prod_{k=1}^{n} \cos(kt) \leq 
    \prod_{k=1}^{n} \left(1 - \frac{(kt)^{2}}{2!} + \frac{(kt)^{4}}{4!}\right),$$ 
    since $k \leq n$ and $|t| \leq 1/n$ implies $kt \leq 1$. 
    When $t = n^{-(3/2-\varepsilon)}$, we have
    $$f_n(t) \le \prod_{k=1}^{n} \left(1-\frac{k^2n^{-(3-2\varepsilon)}}{2} + \frac{k^4n^{-(6-4\varepsilon)}}{24}\right).$$
    To evaluate, we note the terms of this product are all in $[0,1]$, and apply $\log(1-x) \leq -x$:
    $$\log \prod_{k=1}^{n} \left(1- \left(\frac{k^2n^{-(3-2\varepsilon)}}{2}- \frac{k^4n^{-(6-4\varepsilon)}}{24}\right)\right) \leq
    \sum_{k=1}^{n} \left(-\frac{k^2n^{-(3-2\varepsilon)}}{2} + \frac{k^4n^{-(6-4\varepsilon)}}{24}\right).$$
     
    Writing $\sum_{k=1}^{n} k^2 = (1/3+ o(1))n^3$ and $\sum_{k=1}^{n} k^4 = (1/5+o(1))n^5$, we have
    $$\log f_n = -\left(\frac{1}{6}+o(1)\right)n^3n^{-(3-2\varepsilon)} + \left(\frac{1}{120}+o(1)\right)n^5n^{-(6-4\varepsilon)}.$$ Letting $c = 1/6$, 
    $f_n \leq \exp(-c(1+o(1))n^{2\varepsilon} + c(1+o(1))n^{-1+4\varepsilon}),$ 
    and recalling $\varepsilon < 1/4$, 
    $f_n \leq \exp(-c(1+o(1))n^{2\varepsilon})$ as desired.

    \begin{case}$\frac{1}{n} \leq |t| \leq \frac{\pi}{n}$:\end{case}
    Here we use will the monotonicity of $f_n(t)$ in $n$. It follows 
    directly from $f_{n}(t) = \prod_{k=1}^{n} \cos(kt)$ and 
    $0 \leq \cos(x) \leq 1$ that $|f_n(t)| \leq |f_m(t)|$ for $n \geq m$. 
    Let $h_n = \lfloor n/4 \rfloor$ be the greatest integer in $n/4$. Then 
    $|f_n(t)| \leq |f_{h_n}(t)|$. From Case 1, 
    $f_{h_n}(t) \leq \exp(-ch_n^{2\varepsilon}(1+o(1)))$ for 
    $h_n^{-(\frac{3}{2}-\varepsilon)} \leq |t| \leq 1/h_n$. 
    Since $1/n > h_n^{-5/4} \geq h_n^{-(\frac{3}{2}-\varepsilon)}$ for $n > 1050$ and 
    $h_n \leq n/4 \leq n/\pi$ implies $\pi/n \leq 1/h_n$, 
    we get $|f_n(t)| \leq \exp(-ch_n^{2\varepsilon}(1+o(1)))$ for $t \leq \pi/n$ as $n \to \infty$.

    \begin{case} $\frac{\pi}{n} \leq |t| \leq \frac{\pi}{2}$:\end{case}
    Note that it suffices to show that $|f_n(t)| \leq c^{n}$ for a constant $c < 1$, since then
    $$\int_{\pi/n \leq |t| < \pi/2} |f_n(t)| \,dt \leq \pi\cdot c^n = o(n^{-3/2}).$$ To accomplish this, 
    we first transform $f_n(t)$ from a product to a sum using the arithmetic-geometric mean inequality:
    \begin{equation}\label{agmean}
      (f_{n}^2(t))^{1/n} = \left(\prod_{k=1}^{n} \cos^2(kt)\right)^{1/n} \leq \frac{1}{n}\sum_{k=1}^{n} \cos^2(kt). 
    \end{equation}
    The sum on the right-hand side can be simplified as
    \begin{equation}\label{trigsum}
      \sum_{k=1}^{n} \cos^2(kt) = 
      \frac{n}{2} + \frac{1}{2}\sum_{k=1}^n \cos(2kt) = 
      \frac{n}{2} + \frac{\cos((n+1)t)}{2}\frac{\sin(nt)}{\sin(t)}.
    \end{equation}
    Combining equations~\ref{agmean} and~\ref{trigsum}, we can write
    \begin{equation}\label{sinbound}
      |f_n(t)| \leq \left(\frac{1}{2} + \frac{1}{2n}\frac{1}{\sin(t)}\right)^{n/2}.
    \end{equation}
    We will now apply the Jordan-style concavity inequality $|\sin(t)| \geq  \frac{2|t|}{\pi}$ for $0 \leq |t| \leq \pi/2$. For $\pi/n \leq |t| \leq \pi/2$, substitution in equation~\ref{sinbound} gives: 
    \begin{equation*}
      |f_n(t)| \leq \left(\frac{1}{2} + \frac{1}{2n}\frac{\pi}{2|t|}\right)^{n/2} = \left(\frac{1}{2} + \frac{\pi}{4n|t|}\right)^{n/2}.
    \end{equation*}
    Observing that the right-hand side is monotonically decreasing in $|t|$, 
     we have $|f_n(t)| \leq f_n(\pi/n)$. 
    Evaluating, we see 
    \begin{equation*}
      |f_n(t)| \leq \left(\frac{1}{2} - \frac{1}{2n}\right)^{n/2}  
      \end{equation*}
    proving $|f_n(t)| \leq (\sqrt{7/16})^n$ (since we may assume $2n \geq 16$ as $n \to \infty$).

  \end{proof}

  We will also need the following straightforward lemma from analysis.

  \begin{lemma}
    \label{integral_transform}
    Let $c \in \mathbb{R}$ and $a(c), b(c)$ be real-valued functions such that 
    $$\lim_{c \rightarrow \infty} -a(c)\sqrt{c} = 
    \lim_{c \rightarrow \infty} b(c)\sqrt{c} = \infty.$$ 
    Then $$\int_{a(c)}^{b(c)} e^{-ct^2}\,dt \sim \int_{-\infty}^{\infty} e^{-ct^{2}}\,dt$$ as $c \rightarrow \infty$.
  \end{lemma}

\section{Main Result}\label{mainsect}

  We now prove Conjecture 1 holds. 

  \begin{thm}
    When $n \equiv$ 0 or 3 (mod 4), $S(n) \sim \sqrt{6/\pi} \cdot
    2^n n^{-3/2}$. 
  \end{thm}

  \begin{proof}

    When $n \equiv$ 0 or 3 (mod 4), $f_n(t + m\pi) = f_n(t)$ for any integer $m$, so
    \begin{equation}\label{sn_piovertwo}
      S(n) = \frac{2\cdot2^{n-1}}{\pi}\int_{-\frac{\pi}{2}}^{\frac{\pi}{2}}f_n(t)\,dt,
     \end{equation}
    and we may assume $|t| \leq \pi/2$ when evaluating $f_n(t)$.

    By Lemma~\ref{uniform_zero_lemma}, $\int_{n^{-(3/2-\varepsilon)} < |t| < \pi/2} |f_n(t)|\,dt = o(n^{-3/2})$, 
    so it suffices to consider $|t| < n^{-(3/2 - \varepsilon)}$ when estimating $f_{n}(t)$ around $t=0$. 
    Recalling $$f_{n}(t) = \prod_{k=1}^{n}e^{\ln(\cos(kt))},$$ we first use  
    Taylor series to approximate $g_k(t) = \ln(\cos(kt))$ at $t=0$. 
    We have $g_k(t) = -k^{2}t^{2}/2 + R_2$, where $R_2$ is the 
    Lagrange remainder. Then $R_2$ is bounded by a constant times 
    $t^3g_k^{(3)}(t_0)$ for some $t_0$ near $0$. Since 
    $g_k^{(3)}(t) = -2k^3\sin(kt)/\cos^{3}(kt)$, and $t_0$ is small
    (since $|t| < n^{-(3/2 - \varepsilon)}$), we have that 
    $R_2 \leq ak^{3}t^{3}$ where $a$ is constant. The absolute error for
    $g_k(t)$ is thus bounded by $ak^{3}n^{-(9/2 - 3\varepsilon)}$. 

    Around $t=0$, $f_{n}(t)$ can be
    approximated as $\delta\prod_{k=1}^{n}e^{-\frac{k^{2}t^{2}}{2}}$ with error 
    $\delta \leq \prod_{k=1}^{n}e^{ak^{3}n^{-(9/2 - 3\varepsilon)}}$. This
    simplifies to
    \begin{equation}\label{finalapprox}
    f_n(t)\approx e^{-t^{2}/2\sum_{k=1}^{n}k^{2}} = e^{-t^{2}n(n+1)(2n+1)/12}.
    \end{equation} 
    Our error bound simultaneously simplifies to
    $$\delta \leq 
    e^{an^{-(9/2 - 3\varepsilon)}\sum_{k=1}^{n}k^{3}} = e^{an^{-(9/2 - 3\varepsilon)}n^{2}(n+1)^{2}/4}.$$ 
    This proves that the error goes to one as $n$ approaches infinity whenever
    $\varepsilon < \frac{1}{6}$.

    Substituting (\ref{finalapprox}) for $f_n(t)$ in equation~\ref{sn_piovertwo}, and applying 
    Lemma~\ref{uniform_zero_lemma}, we find that 
    
    $$\frac{\pi S(n)}{2^n} = (1+o(1))\int_{-n^{-(3/2-\varepsilon)}}^{n^{-(3/2-\varepsilon)}} e^{-n(n+1)(2n+1)t^{2}/12}\,dt + o(n^{-3/2}).
$$
  
    By Lemma~\ref{integral_transform}, this implies
    $$\frac{\pi S(n)}{2^n} = (1+o(1))\int_{-\infty}^{\infty} e^{-n(n+1)(2n+1)t^{2}/12}\,dt + o(n^{-3/2}).$$
    
    Using
    $$ \int_{-\infty}^{\infty} e^{-Ct^{2}}\,dt = \sqrt{\frac{\pi}{C}}$$ for any constant $C>0$ and 
    $n(n+1)(2n+1) \sim 2n^3$, we have $$S(n) \sim \frac{2^n}{\pi} \sqrt{\frac{12\pi}{2n^{3}}} =
    \sqrt{6/\pi} \cdot 2^n n^{-3/2},$$ as desired.

  \end{proof}

\section{Acknowledgements}

This research was partially completed while the author was an intern at Microsoft Research, Redmond, WA. We thank Henry Cohn for suggesting the problem and providing comments which greatly improved the manuscript. Special thanks are due to Xavier Gourdon, who pointed out a flaw in an earlier version of the paper, and suggested the arithmetic-geometric mean inequality approach now used in the third case of Lemma~\ref{uniform_zero_lemma}.




\bigskip
\hrule
\bigskip

\noindent (Concerned with sequence \seqnum{A025591}.)

\end{document}